\documentclass[11pt]{amsart}

\usepackage{amssymb}
\usepackage{amsmath}
\usepackage{enumerate}
\usepackage[latin1]{inputenc}

\usepackage[T1]{fontenc}
\usepackage[sc]{mathpazo}
\usepackage[pdftex]{graphicx, color}

\definecolor{alert}{rgb}{0.8,0,0}

\newcommand{\s}{\mathbb{S}}

\renewcommand{\r}{\mathbb{R}}

\DeclareMathOperator{\traza}{tr}

\newcommand{\abs}[1]{\left\lvert{#1}\right\rvert}

\newtheorem{theorem}{Theorem}

\newtheorem{corollary}{Corollary}

\theoremstyle{definition}

\theoremstyle{remark}

\numberwithin{equation}{section}

\begin{document}

\title[On stable compact minimal submanifolds]{On stable compact minimal submanifolds}

\author{Francisco Torralbo}
\address{Departamento de Geometr\'{\i}a  y Topolog\'{\i}a \\
Universidad de Granada \\
18071 Granada, SPAIN}
\email{ftorralbo@ugr.es}

\author{Francisco Urbano}
\address{Departamento de Geometr\'{\i}a  y Topolog\'{\i}a \\
Universidad de Granada \\
18071 Granada, SPAIN}
\email{furbano@ugr.es}

\thanks{Research partially supported by a MCyT-Feder research project MTM2007-61775 and the Junta Andalucía Grants P06-FQM-01642 and P09-FQM-4496.}



\date{}
\begin{abstract}
Stable compact minimal submanifolds of the product of a sphere and any Riemannian manifold are classified whenever the dimension of the sphere is at least three. The complete classification of the stable compact minimal submanifolds of the product of two spheres is obtained. Also, it is proved that the only stable compact minimal surfaces of the product of a $2$-sphere and any Riemann surface are the complex ones.
\end{abstract}

\maketitle

\section{Introduction}

The study of the second variation of the volume of minimal submanifolds into Riemannian manifolds can be considered as a classical problem in differential geometry. In fact, the operator of the second variation (the Jacobi operator) carries the information about the stability properties of the submanifold when it is thought as a stationary point for the volume functional. The starting  point could be the paper of Simons, \cite{Si}, where he classified the compact minimal submanifolds of the sphere with the lowest index, proving that there are no stable ones. Later, Lawson and Simons in \cite{LS} characterized the complex submanifolds of the complex projective space as the only stable ones. Ohnita in \cite{Oh} exploited these ideas, classifying  the compact stable minimal submanifolds of the other  compact rank-one symmetric spaces, i.e., the real and quaternionic projective spaces and the Cayley projective plane.

Since then, many works have been devoted to study stability and index of minimal submanifolds in different ambient Riemannian manifolds. In most of these cases one considers two-sided, codimension one  minimal submanifolds, i.e., hypersurfaces with trivial normal bundle, because in this setting the Jacobi operator becomes an operator acting on functions (see \cite{FC}, \cite{FCS}, \cite{DRR} and references therein). For one-sided minimal hypersurfaces and for minimal submanifolds with codimension greater than one, only a few particular situations have been considered (see \cite{MW}, \cite{MU}, \cite{O}, \cite{Oh}, \cite{R} and references therein).

In this paper we come back to the study of stable compact minimal submanifolds with arbitrary codimension, when the ambient manifold is the Riemannian product of a sphere $\s^m(r)$ of radius $r$ and any Riemannian manifold $M$. In this setting, the product of stable minimal submanifolds is a stable minimal submanifold of $\s^m(r)\times M$ (see section $2$ for details). The main contribution in this paper is to prove that the product of stable compact minimal submanifolds of $\s^m(r)\times M$ are the only ones.

\begin{theorem}\label{tm:clasificacion-estables}
Let $M$ be any Riemannian manifold and  $\Phi=(\phi,\psi): \Sigma \rightarrow \s^{m}(r) \times M$ a minimal immersion of a compact $n$-manifold  $\Sigma$, $n\geq 2$, satisfying either $m\geq 3$ or $m=2$ and $\Phi$ is a hypersurface. Then, $\Phi$ is stable if and only if
\begin{enumerate}
	\item $\Sigma=\s^{m}(r)$ and $\Phi(\Sigma)$ is a slice $\s^{m}(r)\times\{q\}$ with $q$ a point of $M$.
\item $\Sigma$ is a covering of $M$ and $\Phi(\Sigma)$ is a slice $\{p\}\times M$ with $p$ a point of $\s^m(r)$.
\item $\psi:\Sigma\rightarrow M$ is a stable minimal submanifold and $\Phi(\Sigma)$ is $\{p\}\times \psi(M)$ with $p$ a point of $\s^m(r)$.
	\item $\Sigma=\s^{m}(r)\times \hat{\Sigma}$, $\Phi=Id\times\psi$, and $\psi:\hat{\Sigma}\rightarrow M$ is a stable minimal submanifold.
\end{enumerate}
\end{theorem}
If $T=\s^1(1)\times\s^1(1)$ and $\Phi:T\rightarrow \s^1(r)\times T$ the totally geodesic embedding given by $\Phi(x,y)=(rxy,x,y)$, where $xy$ denotes the product of the unit complex numbers $x$ and $y$, then $\Phi$ is stable and $\Phi(T)$ is not the product of stable minimal submanifolds. Also, if $\Phi:\s^2(r)\rightarrow\s^2(r)\times\s^2(r)$ is the diagonal map $\Phi(x)=(x,x)$, then $\Phi$ is a stable totally geodesic embedding (see \cite{CU} for details) and $\Phi(\s^2(r))$ is not the product of stable minimal submanifolds. So, in this setting, the above result is the best one.

It is well-known that the complex submanifolds of any Kähler manifold are stable minimal submanifolds, and it is a hard and interesting problem to know which Kähler manifolds have their complex submanifolds as the only stable minimal submanifolds (see \cite{MW}, \cite{SY}). As we mentioned before, this is the case of the complex projective space (see~\cite{LS}). In our setting, the only sphere which admits a Kähler structure is $\s^2(r)$, and then if $M$ is any Riemann surface, we have on $\s^2(r)\times M$ two Kähler structures: $J_1=(J_0,J),\,J_2=(-J_0,J)$, where $J_0$ is the complex structure on $\s^2(r)$ and $J$ the Kähler structure on $M$. Clearly $F(x,y)=(-x,y)$ defines a holomorphic isometry between $(\s^2(r)\times M,J_1)$ and $(\s^2(r)\times M,J_2)$. We prove that the complex compact surfaces of these Kähler surfaces are the only stable compact minimal surfaces.

\begin{theorem}\label{tm:kaeler}
Let $M$ be any Riemann surface and $\Phi=(\phi,\psi):\Sigma\rightarrow \s^2(r)\times M$ a minimal immersion of a compact surface $\Sigma$. Then, $\Phi$ is stable if and only if $\Sigma$ is orientable and $\Phi$ is a complex immersion of the Riemann surface $\Sigma$ into $\s^2(r)\times M$ with respect to one of the two complex structures that $\s^2(r)\times M$ has.
\end{theorem}
As a corollary of these two results, we obtain the complete classification of the stable compact minimal submanifolds of the product of two spheres:
\begin{corollary}\label{cor:esferas}
Let $\Phi=(\phi,\psi): \Sigma \rightarrow \s^{n_1}(r_1) \times \s^{n_2}(r_2)$ be a minimal immersion of a compact $n$-manifold  $\Sigma$, $n\geq 2$. Then, $\Phi$ is stable if and only if one of the following possibilities occurs:
\begin{enumerate}
	\item $\Sigma=\s^{n_1}(r_1)$ and $\Phi(\Sigma)$ is a slice $\s^{n_1}(r_1)\times\{q\}$ with $q$ a point of $\s^{n_2}(r_2)$.
	\item $\Sigma=\s^{n_2}(r_2)$ and $\Phi(\Sigma)$ is a slice $\{p\}\times\s^{n_2}(r_2)$ with $p$ a point of $\s^{n_1}(r_1)$.
	\item $n_1 = n_2 = n= 2$, $\Sigma$ is orientable and $\Phi$ is a complex immersion of the Riemann surface $\Sigma$ in $\s^{2}(r_1)\times\s^{2}(r_2)$ with respect to one of the two complex structures that $\s^2(r_1) \times \s^2(r_2)$ has.
\end{enumerate}
\end{corollary}
In the proof of these results we use the same idea in order to get suitable test normal sections. This idea consists into taking, as test sections, normal components of parallel vector fields of the Euclidean space where the sphere $\s^n(r)$ sits.
\section{Preliminaries}

Let $M_1\times M_2$ be the Riemannian product of two Riemannian manifolds $M_1$ and $M_2$ of dimensions $n_1$ and $n_2$ and $P$ the product structure on $M_1\times M_2$ defined as the tensor
$$
P(v)=P(v_1,v_2)=(v_1,-v_2),\quad v=(v_1,v_2)\in T_{(p,q)} M_1\times M_2.
$$
We note that,
 $$TM_1=\{v\in T(M_1\times M_2)\,|\,Pv=v\},\quad TM_2=\{v\in T(M_1\times M_2)\,|\,Pv=-v\}.$$
It is clear that:
\begin{enumerate}
\item $P$ is a linear isometry,
\item $P$ is parallel, i.e. $\nabla P=0$,
\item $P^2=Id$ and $\traza\, P=n_1-n_2$, where $\traza$ stands for the trace.
\end{enumerate}
The curvature $R$ of $M_1\times M_2$ is given in terms of the curvatures $R_i$ of $M_i,i=1,2$, as follows:
\[
R(v,w)x=(R_1(v_1,w_1)x_1,R_2(v_2,w_2)x_2),
\]
where $v=(v_1,v_2), w=(w_1,w_2)$ and $z=(z_1,z_2)$.

Let $\Phi:\Sigma\rightarrow M_1\times M_2$ be a minimal immersion of a compact $n$-dimensional manifold $\Sigma$. The Jacobi operator $L$ of the second variation is a strongly elliptic operator acting on the space $\Gamma(T^{\bot}\Sigma)$ of sections of the normal bundle of $\Phi$, given by
\[
L=\Delta^{\bot}+\mathfrak{B}+\mathfrak{R},
\]
where $\Delta^{\bot}$ is the second order operator
\[
\Delta^{\bot}=\sum_{i=1}^n\{\nabla^{\bot}_{e_i}\nabla^{\bot}_{e_i}-
\nabla^{\bot}_{\nabla_{e_i}e_i}\},
\]
and $\mathfrak{B}$ and $\mathfrak{R}$ are the endomorphisms defined as follows:
\[
\mathfrak{B}(\eta)=\sigma(e_i,A_{\eta}e_i),\quad \mathfrak{R}(\eta)=(R(\eta,e_i)e_i)^{\bot},
\]
where $\eta\in\Gamma(T^{\bot}\Sigma)$, $\{e_1,\dots,e_n\}$ is an orthonormal reference tangent to $\Sigma$, $\nabla^{\bot}$ is the normal connection, $\sigma$ is the second fundamental form of $\Phi$, $A_{\eta}$ is the shape operator and $\bot$ denotes normal component. The quadratic form associated to $L$ is defined by
\[
Q(\eta)=-\int_{\Sigma}\langle L\eta,\eta\rangle\,d\Sigma.
\]
The {\it index} of $\Phi$ is the index of the quadratic form $Q$ and the immersion $\Phi$ is called {\it stable} if its index is zero, i.e.,
\[
\Phi\quad\hbox{is stable if and only if}\quad Q(\eta)\geq 0,\quad \forall \eta\in\Gamma(T^{\bot}\Sigma).
\]
If $\Phi_i:\Sigma_i\rightarrow M_i$, are minimal immersions of the compact manifolds $\Sigma_i,\,i=1,2$, then $\Phi=\Phi_1\times \Phi_2:\Sigma_1\times \Sigma_2\rightarrow M_1\times M_2$ is also a minimal immersion. Moreover, $\Gamma(T^{\bot}\Sigma_i)$ is a subspace of $\Gamma(T^{\bot}(\Sigma_1\times\Sigma_2))$ and if $\eta_i\in\Gamma(T^{\bot}\Sigma_i)$, then it is easy to check that
\[
Q(\eta_1)=Q_1(\eta_1)\,\hbox{volume}\,(\Sigma_2),\quad Q(\eta_2)=Q_2(\eta_2)\,\hbox{volume}\,(\Sigma_1),
\]
where $Q_i$ are the quadratic forms associated to $\Phi_i,\,i=1,2$.

On the other hand, given a normal section  $\eta\in\Gamma(T^{\bot}(\Sigma_1\times\Sigma_2))$, for any $(p,q)\in\Sigma_1\times\Sigma_2$, we have that
\[
\eta(p,q)=(\eta_1(p,q),\eta_2(p,q)):=((\eta_1)_q(p),(\eta_2)_p(q)).
\]
It is clear that $(\eta_1)_q\in\Gamma(T^{\bot}\Sigma_1)$ for any $q\in\Sigma_2$, and that  $(\eta_2)_p\in\Gamma(T^{\bot}\Sigma_2)$ for any $p\in\Sigma_1$. Moreover, it is straightforward to check that
\[
Q(\eta)=\int_{\Sigma_1}Q_2((\eta_2)_p)\,d\Sigma_1+\int_{\Sigma_2}Q_1((\eta_1)_q)\,d\Sigma_2.
\]
Hence we have proved that:

{\it $\Phi_1\times\Phi_2:\Sigma_1\times\Sigma_2\rightarrow M_1\times M_2$ is stable
if and only if $\Phi_i:\Sigma_i\rightarrow M_i,\,i=1,2,$ are stable.}
\section{Proofs of the results.}
\begin{proof}[Proof of Theorem~\ref{tm:clasificacion-estables}]
As the examples appearing in (1), (2), (3) and (4) are product of stable minimal submanifolds (some of them of dimension or codimension zero), they are stable minimal submanifolds.

 Conversely,  let $\Phi:\Sigma\rightarrow \s^{m}(r)\times M$ be a stable minimal immersion of codimension $p$. We consider $\s^{m}(r) \subseteq \r^{m+1}$. Given a vector $a \in \r^{m+1}$, its normal component to $\Phi$ (respectively its tangential component to $\Sigma$) will be denote by $\eta_a$ (respectively $X_a$). We shall use $\eta_a$ as test section, hence we next compute $L(\eta_a)$.

Let $\{e_i:\, i = 1, \ldots, n\}$  be a local orthonormal reference in $T\Sigma$ and $\{\xi_\alpha:\, \alpha = 1, \ldots, p\}$ a local orthonormal reference in $T^\bot \Sigma$. We will denote by $A$ the square $n$-matrix $A_{ij}=\langle Pe_i,e_j\rangle$ and by $B$ the square $p$-matrix $B_{\alpha\beta}=\langle P\xi_{\alpha},\xi_\beta\rangle$.

Deriving the vector $a\in\r^{m+1}$ with respect to $e_i$  and taking normal and tangential components to $\Sigma$ we obtain that
\begin{align}
\nabla^\bot_{e_i} \eta_a &= -\sigma(X_a, e_i) -\frac{\langle \phi,a\rangle}{2r^2}\sum_{\alpha=1}^p\langle Pe_i,\xi_{\alpha}\rangle \xi_{\alpha}, \label{eq:parte-normal-derivada-X_a}\\
\nabla_{e_i} X_a &= A_{\eta_a} e_i -\frac{\langle \phi,a\rangle}{2r^2}(e_i+\sum_{j=1}^nA_{ij} e_j).  \label{eq:parte-tangente-derivada-X_a}
\end{align}
Deriving again~\eqref{eq:parte-normal-derivada-X_a} with respect to $e_i$ and using that $\phi_*(v)=(v+Pv)/2$ for any $v\in T\Sigma$, we obtain
\[
\begin{split}
\Delta^{\bot}\eta_a =&-\sum_{i=1}^n\{(\nabla\sigma)(e_i,X_a,e_i)+\sigma(e_i,\nabla_{e_i}X_a)\}\\
&-\frac{1}{4r^2}\sum_{i,\alpha}\langle e_i+Pe_i,a\rangle\langle Pe_i,\xi_{\alpha}\rangle\xi_{\alpha}-\frac{\langle\phi,a\rangle}{2r^2}\sum_{i,\alpha}\langle Pe_i,\xi_{\alpha}\rangle\nabla^{\bot}_{e_i}\xi_{\alpha}\\
&-\frac{\langle\phi,a\rangle}{2r^2}\sum_{i,\alpha}\langle Pe_i,-A_{\xi_{\alpha}}e_i+\nabla^{\bot}_{e_i}\xi_{\alpha}\rangle\xi_{\alpha}.
\end{split}
\]
Now using the Codazzi equation, ~\eqref{eq:parte-tangente-derivada-X_a} and the definition of the Jacobi operator, we get
\[
\begin{split}
L\eta_a=&\sum_{i=1}^n(\bar{R}(X_a+\eta_a,e_i)e_i)^{\bot}+\frac{\langle\phi,a\rangle}{2r^2}\sum_{i,j=1}^nA_{ij}\sigma(e_i,e_j)\\
&-\frac{1}{4r_1^2}\sum_{i,\alpha}\langle e_i+Pe_i,a\rangle\langle Pe_i,\xi_{\alpha}\rangle\xi_{\alpha}-\frac{\langle\phi,a\rangle}{2r^2}\sum_{i,\alpha}\langle Pe_i,\xi_{\alpha}\rangle\nabla^{\bot}_{e_i}\xi_{\alpha}\\
&-\frac{\langle\phi,a\rangle}{2r_1^2}\sum_{i,\alpha}\langle Pe_i,-A_{\xi_{\alpha}}e_i+\nabla^{\bot}_{e_i}\xi_{\alpha}\rangle\xi_{\alpha}.
\end{split}
\]
From the expression of the curvature $R$ and as $P(X_a+\eta_a)=X_a+\eta_a$, it is easy to check that
\[
\sum_{i=1}^n(\bar{R}(X_a+\eta_a,e_i)e_i)^\bot=\frac{n+\traza A}{2r^2}\eta_a-\sum_{i,\alpha}\frac{\langle P\xi_{\alpha},e_i\rangle\langle e_i,a\rangle}{2r^2}\xi_{\alpha}.
\]

If $\{a_1,\dots,a_{m+1}\}$ is an orthonormal basis of $\r^{m+1}$, then using the above two formulas we obtain that
\[
\begin{split}
\sum_{k=1}^{m+1}\langle L\eta_{a_k},\eta_{a_k}\rangle
&=\frac{n+\traza A}{2r^2}\sum_{\alpha,k}\langle a_k,\xi_{\alpha}\rangle^2-\sum_{i,\alpha,k}\frac{\langle P\xi_{\alpha},e_i\rangle\langle e_i,a_k\rangle\langle\xi_{\alpha},a_k\rangle}{2r^2}\\
&\quad-\frac{1}{4r^2}\sum_{i,\alpha,k}\langle Pe_i,\xi_{\alpha}\rangle\langle e_i+Pe_i,a_k\rangle\langle\xi_{\alpha},a_k \rangle\\
&=\frac{n+\traza A}{8r_1^2}\sum_{\alpha=1}^p|\xi_{\alpha}+P\xi_{\alpha}|^2
-\sum_{i,\alpha}\frac{\langle P\xi_{\alpha},e_i\rangle\langle e_i+Pe_i,\xi_{\alpha}+P\xi_{\alpha}\rangle}{4r^2}\\
&=\frac{(n+\traza A)(p+\traza B)}{4r_1^2}-\frac {1}{2r^2}\sum_{i,\alpha}\langle Pe_i,\xi_{\alpha}\rangle^2.
\end{split}
\]
Since $\traza A+\traza B=\traza P=2m-n-p$ and
\[
n=\sum_{i=1}^n|Pe_i|^2=\sum_{i,j=1}^n\langle Pe_i,e_j\rangle^2+\sum_{i,\alpha}\langle Pe_i,\xi_{\alpha}\rangle^2=\traza(A^2)+\sum_{i,\alpha}\langle Pe_i,\xi_{\alpha}\rangle^2,
\]
we obtain that
\[
\sum_{k=1}^{m+1}\langle L\eta_{a_k},\eta_{a_k}\rangle
=\frac{1}{4r^2}\left[ 2\traza(A^2) - (\traza A)^2 + 2(m - n)\traza{A} + n(2m - n - 2)\right].
\]
The stability of $\Phi$ implies that $0\leq\sum_{k}Q(\eta_{a_k})=-\sum_{k}\int_{\Sigma}\langle L\eta_{a_k},\eta_{a_k}\rangle$, so we finally get that
\begin{equation}
0\leq \int_{\Sigma}\left[- 2\traza(A^2) + (\traza A)^2 - 2(m - n)\traza{A} - n(2m - n - 2)\right]d\Sigma.  \label{eq:desigualdad estabilidad}
\end{equation}

To use this stability inequality, we will consider three different cases:

\medskip

\noindent\textbf{First case}: $n\leq m$ and $m\geq 3$.

The Schwarz inequality implies that $(\traza A)^2\leq n\traza(A^2)$ and the equality holds if and only if $A=\lambda Id$ for certain function $\lambda$ on $\Sigma$. So ~\eqref{eq:desigualdad estabilidad} becomes
\begin{equation}
0\leq \int_{\Sigma}(\traza A+n)\left(\frac{n-2}{n}\traza A-2m+n+2\right)\,d\Sigma. \label{eq:desigualdad estabilidad-f}
\end{equation}
As $-n\leq \traza A\leq n$ and $m\geq 3$, the integrand is non-negative, hence the equality holds in the above inequality. This means that $A=\lambda Id$ and either $\traza A=-n$ or $\traza A=n=m$. So we have that either $A=- Id$ or $A= Id$ and $n=m$.

\medskip

\noindent\textbf{Second case}: $n>m$ and $m\geq 3.$

In this setting, for any point $x\in\Sigma$, we have that $\dim\ker d\phi_x\geq n-m$ and so, there exists an $(n-m)$-dimensional linear subspace $V_x\subset\ T_x\Sigma$ such that $Pv=-v$ for any $v\in V_x$. Hence, we can decompose $T_x\Sigma=V_x\oplus Z_x$, with $Z_x$ orthogonal to $V_x$ and $\dim Z_x=n-(n-m)=m$. Thus, our matrix $A$ can be written as $A=-Id\oplus \hat{A}$, with $\hat{A}_{ij}=\langle Pz_i,z_j\rangle$, and $\{z_1,\dots,z_p\}$ being
an orthonormal basis of $Z_x$. In particular,
\[
\traza A=m-n+\traza \hat{A},\quad\traza (A^2)=n-m+\traza(\hat{A}^2).
\]
 Then ~\eqref{eq:desigualdad estabilidad} becomes
\[
0\leq\int_{\Sigma}\bigl( -2\traza (\hat{A}^2)+(\traza \hat{A})^2-m(m-2) \bigr)d\Sigma.
\]
The Schwarz inequality implies that $(\traza \hat{A})^2\leq m\traza(\hat{A}^2)$ and the equality occurs if and only if $\hat{A}=\lambda Id$ for certain function $\lambda$ on $\Sigma$. So the above integral inequality transforms into
\begin{equation}
0\leq \frac{m-2}{m}\int_{\Sigma}\bigl((\traza \hat{A})^2-m^2\bigr)d\Sigma. \label{eq:inequality-hat}
\end{equation}
As $m\geq 3$ and $-m\leq \traza \hat{A}\leq m$, the integrand of~\eqref{eq:inequality-hat} is non-positive and we obtain the equality in~\eqref{eq:inequality-hat}, which means that $\traza\hat{A}=\pm m$ and $\hat{A}=\lambda Id$. So, $\hat{A}=\pm Id$. Hence we obtain that either $A=-Id$ or $A=-Id\oplus Id$.

\medskip
\noindent\textbf{Third case}: $p=1$ and $m=2$.

Following the argument used to get ~\eqref{eq:desigualdad estabilidad} and changing the matrix $A$ by $B$, it is straightforward to get a second version of the stability inequality
\[
0\leq \int_{\Sigma}\left[- 2\traza(B^2) + (\traza B)^2 - 2(m - p)\traza{B} - p(2m - p - 2)\right]d\Sigma.
\]
If the codimension $p=1$, then the matrix $B=\lambda$ for certain function $\lambda$. Then, taking $m=2$, the above inequality become in
\[
0\leq -\int_{\Sigma}(\lambda+1)^2\,d\Sigma.
\]
Since the integrand is non-negative, the equality holds in the above inequality, which means that $\lambda=-1$.

On the other hand, if $n=2$, then the equality holds in~\eqref{eq:desigualdad estabilidad-f}, and so, $A=\lambda Id$. Hence $1=\traza P=\traza A+\traza B=2\lambda-1$ and we obtain that $A=Id$.

If $n>2=m$, then the equality holds in~\eqref{eq:inequality-hat}, and so $\hat{A}=\lambda Id$. Hence $3-n=\traza P=\traza A+\traza B=2-n+2\lambda-1$ and we obtain that $\hat{A}=Id$. So $A=-Id\oplus Id$.

In summary, if $\Phi$ is stable, the matrix $A$ has only three possibilities:
\[
A=-Id,\quad A=Id\quad n=m,\quad A=-Id\oplus Id.
\]

When $A=- Id$, we have that $Pe_i=-e_i+ (Pe_i)^{\bot}$, for $1\leq i\leq n$. As $P$ is an isometry, $(Pe_i)^{\bot}=0$ , and hence $Pv=-v$ for any $v\in T\Sigma$. This means that $\phi_*(v)=0$ for any $v\in T\Sigma$ and so, $\phi$ is a constant map. Hence $\psi:\Sigma\rightarrow M$ is either an stable minimal immersion if $n<\hbox{dim}\,M$ or $\Sigma$ is a covering of $M$ if $n=\hbox{dim}\, M$. We obtain the cases (3) and (2).

 When $n=m$ and $A=Id$, using a similar argument as before, $\psi:\Sigma\rightarrow M$ is constant and $\phi:\Sigma\rightarrow \s^{n}(r)$ is the identity. We obtain the case (1).

The last possibility means that we have two orthogonal distributions $D_1$ and $D_2$ on $\Sigma$ of dimensions $m$ and $n-m$ respectively, defined by
\[
D_1=\{v\in T\Sigma\,|\,Pv=v\}\quad D_2=\{v\in T\Sigma\,|\,Pv=-v\},
\]
such that $T\Sigma=D_1\oplus D_2$. In fact, $\phi:\Sigma\rightarrow \s^{m}(r)$ is a Riemannian submersion for which $D_1$ defines the vertical subspaces and $D_2$ defines the horizontal ones. As $P$ is parallel, these distributions are totally geodesic foliations on $\Sigma$, and so $\Sigma=\s^{m}(r)\times \Sigma_2$ and $\Phi=Id\times\Phi_2$ where $\Phi_2:\Sigma_2\rightarrow M$ is a stable submanifold. We obtain the case (4).
\end{proof}

\begin{proof}[Proof of Theorem~\ref{tm:kaeler}]
As $n=m=2$, we get the equality in ~\eqref{eq:desigualdad estabilidad-f} and then the matrix $A=\lambda Id$, for certain function $\lambda$ on $\Sigma$. In this setting, if $\omega_i$ is the Kähler $2$-forms associated to the complex structures $J_i$, of $\s^2(r)\times M$, $i=1,2$, it is straightforward to check that
\[
(A_{11}-A_{22})^2 + 4A_{12}^2 = (2 - \abs{\Phi^*\omega_1}^2)(2 - \abs{\Phi^*\omega_2}^2).
\]
As $A=\lambda Id$, then $(2 - \abs{\Phi^*\omega_1}^2)(2 - \abs{\Phi^*\omega_2}^2)=0$. But the minimal surfaces of Kähler surfaces are either complex or the set of complex points are isolated (see~[W]), from where we obtain that either
$\abs{\Phi^*\omega_1}^2=2$ or $\abs{\Phi^*\omega_2}^2=2$. This means that $\Phi$ is a complex immersion with respect to either $J_1$ or $J_2$.
\end{proof}

\end{document}